\documentclass[12pt]{article}
\usepackage[a4paper,margin=1in]{geometry}
\usepackage{amsmath,amssymb,amsthm,mathtools}
\usepackage{xcolor}
\usepackage[colorlinks=true,linkcolor=blue!45!black,citecolor=blue!45!black,urlcolor=blue!45!black]{hyperref}
\newtheorem{theorem}{Theorem}
\newtheorem{lemma}[theorem]{Lemma}
\newtheorem{proposition}[theorem]{Proposition}
\theoremstyle{remark}

\newcommand{\Q}{\mathbb Q}
\newcommand{\R}{\mathbb R}
\newcommand{\C}{\mathbb C}
\newcommand{\Z}{\mathbb Z}

\newcommand{\OK}{\mathcal O_K}
\newcommand{\OF}{\mathcal O_F}
\newcommand{\N}{\operatorname{N}}
\newcommand{\Cl}{\operatorname{Cl}}
\newcommand{\covol}{\operatorname{covol}}
\newcommand{\Span}{\operatorname{span}}
\newcommand{\Res}{\operatorname{Res}}

\newcommand{\doi}[1]{\href{https://doi.org/#1}{\nolinkurl{#1}}}

\title{\Large The existence of small defining polynomials for number fields}

\author{Manjul Bhargava and Sameera Vemulapalli}

\date{\today}

\begin{document}

\maketitle

\begin{abstract}
We prove that every number field of degree $n$ and absolute
discriminant $D$ has a defining integral polynomial of degree $n$ whose
coefficients are all $O_n(\sqrt D)$ in size, thus resolving a 1998 conjecture of Ruppert.

Our proof compares ideal lattices with their trace duals, via which we handle all number fields
with the exception of CM fields whose maximal totally real
subfields have small discriminant. We then treat the remaining fields by
averaging completed Epstein zeta functions over all possible ideal classes
and archimedean scalings,  for which we make critical use of Hecke's formula and the
subconvexity theorem of Michel and~Venkatesh.  

Our proof does not use the Quasi-Riemann Hypothesis. 
\end{abstract}

\section{Introduction}

The purpose of this article is to prove the following theorem:

\vspace{.065in}
\begin{theorem}\label{thm:main}
Let $K$ be a number field of degree $n$ and absolute discriminant $D_K$.
Then there exists a polynomial
\[
 f(x)=a_0x^n+a_1x^{n-1}+\cdots+a_n\in\Z[x]
\]
of degree $n$ with a root $\alpha$ such that $K=\Q(\alpha)$ and
$H(f):=\max \{|a_i|\} \ll_n\sqrt{D_K}$.
\end{theorem}
\vspace{.025in}

The statement of Theorem~\ref{thm:main} was conjectured by
Ruppert~\cite[Question~2]{Ruppert}. 
%Equivalently, the statement is equivalent to: every number
%field $K$ of degree $n$ should have a generator of absolute
%multiplicative Weil height $O_n(D_K^{1/(2n)})$.
Ruppert himself proved the quadratic case, using an elementary
argument for real quadratic fields and Duke's equidistribution
theorem for Heegner points~\cite{Duke} to handle imaginary quadratic
fields. He also proved the result for totally real fields of
prime degree. Vaaler and Widmer~\cite{VW13}
subsequently proved the conjecture whenever $K$ has
a real embedding. They also
proved it for arbitrary number fields under the generalized
Riemann hypothesis for the Dedekind zeta function of the normal
closure, by applying an effective form of the Chebotarev
density theorem.

Further progress was made by Pierce, Turnage-Butterbaugh,
and Wood \cite{PTBW}, who obtained the bound for almost all
fields in the families covered by their effective Chebotarev
theorems, under the hypotheses of those theorems.
Widmer \cite{WidmerAbelian} proved the conjecture for abelian
number fields of sufficiently large degree and showed that
the exponent $1/2$ in Theorem~\ref{thm:main} cannot be decreased
in any even degree. Akhtari, Vaaler, and Widmer \cite{AVW}
proved the bound for fields containing a non-real root of unity,
with an integral generator, and showed that possible exceptions
must be totally complex fields that are Galois over their
maximal totally real subfields and have no roots of unity
other than $\pm1$.

Recently, and independently, Frei, Sofos, and Widmer
\cite[Corollary~1.7]{FSW}, motivated by the recent work on the 
quasi-Riemann Hypothesis (QRH) by OpenAI~\cite{openai} and its extension to 
Dedekind zeta functions by Cureton and Thorner~\cite[Theorem 1.1]{ct}, proved the bound in Theorem~\ref{thm:main} assuming QRH for
the Dedekind zeta function of the normal closure of $K$. 
Their argument extends the Chebotarev approach of Vaaler and
Widmer: a zero-free half-plane $\Re(s)>1-\delta$, for any fixed
$\delta>0$, suffices. In contrast, our proof does not use~QRH, but instead 
uses the geometry of ideal lattices together with Hecke's formula and the
subconvexity theorem of Michel and Venkatesh \cite{MV}.

\smallskip
Our strategy for proving Theorem~\ref{thm:main} is as follows.
We view each fractional ideal $I$ of $\OK$ as a lattice in $\R^n$ through
the archimedean embeddings and rescale it to have covolume one.
Two elements $\alpha,\beta\in I$ of normalized length at most $B$
give an integral polynomial
$f(x)=\N(\alpha x-\beta)/\N(I)$ of height at most $(2B)^n\sqrt{D_K}$.
We must therefore find such a pair with $B\ll_n1$ and
$\Q(\beta/\alpha)=K$.

The key geometric idea is to use both an ideal lattice and its
dual under the trace pairing. 
%The latter is obtained as 
%$I^\vee=\mathfrak d_K^{-1}I^{-1}$.
The inequalities for a lattice and its dual imply that one of
them contains $\lceil n/2\rceil$ independent vectors of length
at most $2^n/\omega_n$, where $\omega_n$ is the volume of the
unit ball. If~$K$ has no index-two subfield, bounded
integer combinations of these vectors give a ratio generating
$K$ and hence the required polynomial.

Suppose this construction fails for every ideal and every
positive rescaling of the embeddings that preserves volume.
Then $n=2m$, and the short elements of each ideal span a fixed
one-dimensional space over a subfield $F$ of degree $m$.
This space is unchanged by multiplication by any unit of $K$.
Dirichlet's unit theorem then forces $F$ to be totally real
and $K/F$ to be CM. The same restriction on the short elements
gives an integral representative of norm $\ll_n\sqrt d$ in
every class of $\Cl(K)/\operatorname{im}\Cl(F)$, where
$d=D_K/D_F^2$. Counting these representatives and comparing
with the class-number formula and a Brauer--Siegel-type estimate
shows that we must have $D_F\ll_{n,\epsilon}D_K^\epsilon$ for every $\epsilon>0$.
%The small discriminant of $F$ thus follows from the restriction on relative ideal classes.

For such CM extensions, our polynomial $f$ is the normalized norm
of a totally-positive-definite binary quadratic form over $F$:
\[
 Q(x,y)=\N_{K/F}(\alpha x-\beta y)
       =\alpha\bar\alpha x^2
        -(\alpha\bar\beta+\bar\alpha\beta)xy
        +\beta\bar\beta y^2,\qquad
 f(x)=\frac{\N_{F/\Q}(Q(x,1))}{\N(I)}.
\]
Here the bar denotes the nontrivial automorphism of $K/F$.
When $\beta/\alpha$ generates $K$, the form $Q$ is positive
definite at every real embedding of $F$. Bounds for the
normalized lengths of $\alpha$ and $\beta$ bound the
coefficients of these conjugate quadratic forms, and taking
their product then gives the desired bound for $H(f)$.

The familiar case $F=\Q$ illustrates the geometry. A reduced
positive definite form $ax^2+bxy+cy^2$ of discriminant $-D$
satisfies $|b|\leq a\leq c$ and $4ac=b^2+D$.
If $a\geq c_0\sqrt D$ for a fixed $c_0>0$, then
$a\leq\sqrt{D/3}$ and $c\leq\sqrt D/(3c_0)$, so all its
coefficients are $O_{c_0}(\sqrt D)$.
For general $F$, we use the corresponding lattice criterion.
If a covolume-one lattice has shortest-vector length at least
$r>0$, Minkowski's second theorem bounds its successive minima by
\[
 r^{n-1}\lambda_n
 \leq\lambda_1\cdots\lambda_n\leq\frac{2^n}{\omega_n}.
\]
It therefore has $n$ independent vectors of length at most
$2^n/(\omega_n r^{n-1})$. We seek an ideal and a scaling for
which this holds with $r=r_m>0$ depending only on $m$.

To show that these exist, we use the completed Epstein zeta function $Z_L$.
Poisson summation shows that $Z_L(1/2)\geq1$ whenever the
normalized lattice $L$ has a nonzero vector of length at most
an explicit $r_m$. We average this central value, giving each
ideal class equal weight and integrating over volume-preserving
scalings modulo units with probability measure.
Hecke's formula evaluates the average as
$c_mD_K^{-1/4}\zeta_K(1/2)/\kappa_K$, where $c_m>0$ depends
only on $m$ and $\kappa_K=\Res_{s=1}\zeta_K(s)$.

The subconvexity theorem of Michel and Venkatesh \cite{MV} gives
$|\zeta_K(1/2)|\ll_m D_F^a d^{1/4-\delta}$ for some
$a\geq0$ and $\delta>0$ depending only on $m$.
The saving $\delta$ is essential here: it produces a negative
power of $D_K$ after multiplication by $D_K^{-1/4}$.
Michel and Venkatesh prove that the constant in their estimate
depends at most polynomially on $D_F$ when $[F:\Q]$ is fixed.
Since the relative class-group argument above gives
$D_F\ll_{n,\epsilon}D_K^\epsilon$, and the residue estimate
gives $\kappa_K^{-1}\ll_{n,\epsilon}D_K^\epsilon$, these factors
can be absorbed into the power saving. The average consequently
tends to zero. For sufficiently large $D_K$, some lattice in
the average must then have shortest-vector length greater than
$r_m$, completing the construction.

\smallskip

This paper is organized as follows.  In Section~\ref{sec:polynomial}, we establish integrality and
coefficient bounds for the norm polynomial and explain how to
choose a generating ratio.
In Section~\ref{sec:CM}, we compare an ideal lattice with its
dual and prove the CM reduction.
In Section~\ref{sec:base}, we count relative ideal classes to
bound $D_F$.
In Section~\ref{sec:forms}, we give the quadratic-form
interpretation and the short-vector criterion for $Z_L(1/2)$.
In Section~\ref{sec:uniform}, we apply Hecke's formula
\cite{Hecke,Khayutin} and Michel--Venkatesh
\cite[Theorem~1.2]{MV} to bound the average.
Section~\ref{sec:conclusion} combines these estimates to prove
Theorem~\ref{thm:main}.

\section{Polynomials from short ideal elements}\label{sec:polynomial}

The case $n=1$ follows by taking $f(x)=x$. Henceforth let $n\geq2$,
write $D=D_K$, and put $M=\binom n2$.
By a defining polynomial for $K$, we mean an integer polynomial
of degree $n$ with a root that generates $K$ over $\Q$.

Let $r_1$ be the number of real embeddings of $K$ and $r_2$ the
number of conjugate pairs of complex embeddings, so that
$n=r_1+2r_2$. Choose one embedding from each complex pair and write
\[
 \sigma:K\longrightarrow\R^{r_1}\times\C^{r_2}.
\]
We identify each copy of $\C$ with $\R^2$ and use the usual Euclidean
metric. Thus the embeddings themselves, without any additional
factors, realize each fractional ideal $I$ as a lattice in $\R^n$.
Write $\|\alpha\|=\|\sigma(\alpha)\|$, and write $\N(I)$ for the
absolute norm of an ideal and $\N(\alpha)=\N_{K/\Q}(\alpha)$ for
the norm of an element. Put $d_v=1$ at a real place and $d_v=2$
at a complex place. Then
\[
 \covol(\sigma(I))=2^{-r_2}\sqrt D\,\N(I).
\]

\begin{lemma}\label{boundlemma}
For any fractional ideal $I$ of $\OK$ and any two elements
$\alpha,\beta\in I$, write
\[
 F_{I;\alpha,\beta}(x,y)
 =\frac{\N(\alpha x+\beta y)}{\N(I)}
 =\sum_{i=0}^n c_i x^iy^{n-i}.
\]
Then $F_{I;\alpha,\beta}$ is integral, and
\[
 |c_i|\leq\binom ni\frac{\|\alpha\|^i\|\beta\|^{n-i}}{\N(I)}.
\]
Moreover, its coefficient height satisfies
\begin{equation}\label{eq:binary-bound}
 H(F_{I;\alpha,\beta})
 \leq\frac{2^n}{\N(I)}
       \prod_v\max\{|\sigma_v(\alpha)|,|\sigma_v(\beta)|\}^{d_v}.
\end{equation}
\end{lemma}

\begin{proof}
Multiplication by $\alpha x+\beta y$ defines a map from
$\OK[x,y]$ to $I[x,y]$. With respect to integral bases of $\OK$
and $I$, its matrix has entries in $\Z[x,y]$, and its determinant
is $F_{I;\alpha,\beta}$ up to sign. This yields integrality.

The norm is the product of $\sigma(\alpha)x+\sigma(\beta)y$ over
all $n$ embeddings $\sigma:K\hookrightarrow\C$. The coefficient
of $x^iy^{n-i}$ is therefore a sum of $\binom ni$ products,
each with $i$ conjugates of $\alpha$ and $n-i$ conjugates of
$\beta$. Every conjugate is bounded in absolute value by the
Euclidean length of the corresponding element. This proves
the first bound. Bounding each factor instead by
$\max\{|\sigma(\alpha)|,|\sigma(\beta)|\}$ gives
\eqref{eq:binary-bound}.
\end{proof}

We will also allow positive changes of scale in the archimedean
coordinates. For $u=(u_v)_v$ with $\prod_vu_v^{d_v}=1$, put
\begin{equation}\label{eq:lattice}
 t_I=(2^{-r_2}\sqrt D\,\N(I))^{-1/n},
 \qquad L(I,u)=t_Iu\sigma(I).
\end{equation}
This lattice has covolume one. We call
$\|t_Iu\sigma(\alpha)\|$ the normalized length of $\alpha\in I$.
If $\alpha$ and $\beta$ have normalized lengths at most $B$,
Lemma~\ref{boundlemma} gives
\begin{equation}\label{eq:norm}
 \Phi_{I;\alpha,\beta}(x)
 :=\frac{\N(\alpha x-\beta)}{\N(I)}\in\Z[x],
 \qquad H(\Phi_{I;\alpha,\beta})\leq(2B)^n\sqrt D.
\end{equation}
Indeed, the factors $u_v$ cancel in the product in
\eqref{eq:binary-bound}, and
$t_I^{-n}/\N(I)=2^{-r_2}\sqrt D$.

We next explain how to choose a short numerator and denominator
whose ratio generates $K$.

\begin{lemma}\label{lem:ratios}
Let $\alpha_1,\ldots,\alpha_k\in I$, with $\alpha_1\ne0$ and
$k\leq n$, have normalized lengths at most~$R$. Suppose that
$\Q(\alpha_2/\alpha_1,\ldots,\alpha_k/\alpha_1)=K$.
Then there is an integer linear combination $\beta=\sum_i a_i\alpha_i$,
with $0\leq a_i\leq M$, such that $\beta/\alpha_1$ generates $K$.
In particular, there is a defining polynomial $f$ for $K$ with
\begin{equation}\label{eq:ratios}
 H(f)\leq(2nMR)^n\sqrt D.
\end{equation}
\end{lemma}

\begin{proof}
For each pair of distinct embeddings $\sigma,\tau:K\hookrightarrow\C$,
the equation
\[
 \sum_{i=1}^k a_i
 \bigl(\sigma(\alpha_i/\alpha_1)-\tau(\alpha_i/\alpha_1)\bigr)=0
\]
is a nonzero linear equation in the $a_i$; otherwise $\sigma$
and $\tau$ would agree on generators of $K$.
Each such equation contains at most $(M+1)^{k-1}$ points of
$\{0,\ldots,M\}^k$: after fixing all but a coordinate with nonzero
coefficient, there is at most one choice for that coordinate.
There are $M$ pairs of embeddings, so their union does not exhaust
this set.

Choose $(a_1,\ldots,a_k)$ outside the union and put
$\beta=\sum_i a_i\alpha_i$. Then all conjugates of
$\beta/\alpha_1$ are distinct, so this ratio generates $K$.
The normalized length of $\beta$ is at most $nMR$.
Taking $f=\Phi_{I;\alpha_1,\beta}$ in \eqref{eq:norm} proves
\eqref{eq:ratios}.
\end{proof}

For a Euclidean lattice $L$, let $\lambda_j(L)$ be its $j$th
successive minimum, i.e., the least radius of a ball centered at zero
containing $j$ independent lattice vectors. Independence of
ideal elements over $\Q$ is equivalent to independence of their
Minkowski vectors over $\R$, since their coordinates in an
integral basis are rational.

\section{Reduction to CM fields}\label{sec:CM}

We first prove the theorem when $K$ has no index-two subfield.
The proof uses only an ideal of $\OK$ and its dual. We then
vary the archimedean scalings to identify the fields with index-2 subfields
for which this construction can fail.

For later use, we record the relevant inequalities for a lattice and its dual.
Let $\omega_n$ be the volume of the Euclidean unit ball and
put $T_n=(2^n/\omega_n)^2$. If
$L^*=\{w:\langle w,L\rangle\subset\Z\}$ is the Euclidean dual of a
full lattice $L$, then
\begin{equation}\label{eq:transference}
 1\leq\lambda_j(L)\lambda_{n+1-j}(L^*)\leq T_n
 \qquad(1\leq j\leq n).
\end{equation}
To prove the lower bound, suppose that the product were less than
one. The integer inner products between $j$ independent short vectors
in $L$ and $n+1-j$ independent short vectors in $L^*$ would all have
absolute value less than one, and hence would vanish. Their spans
would then be orthogonal, although their dimensions sum to $n+1$,
which is impossible. For the upper bound, Minkowski's second theorem
applied to $L$ and $L^*$ bounds the product of all $n$ quantities in
\eqref{eq:transference} by $T_n$. Each quantity is at least one, so
each is also at most $T_n$.

Let $\mathfrak d_K$ be the different ideal of $\OK$, let $C$ conjugate each
complex coordinate, and let $A$ multiply the $v$th coordinate by
$d_v$. With the ordinary Euclidean inner product, we have
\[
 \langle\sigma(\alpha),AC\sigma(\beta)\rangle
       =\operatorname{Tr}_{K/\Q}(\alpha\beta).
\]
Thus $\sigma(I)^*=AC\sigma(\mathfrak d_K^{-1}I^{-1})$.
Since $\N(\mathfrak d_K)=D$, the normalized dual is
\begin{equation}\label{eq:duality}
 L(I,u)^*=C L(\mathfrak d_K^{-1}I^{-1},u^*),\qquad
 u_v^*=2^{-2r_2/n}d_vu_v^{-1}.
\end{equation}
Indeed, the normalizing scalars for $I$ and
$\mathfrak d_K^{-1}I^{-1}$ have product $2^{2r_2/n}$.
Moreover,
\[
 \prod_v(u_v^*)^{d_v}=2^{-2r_2}\prod_vd_v^{d_v}=1.
\]
Hence the dual is isometric to another lattice in our family.
The map $C$ acts on the ambient real space; it need not arise from
an automorphism of $K$.

\begin{proposition}\label{thminmostcases}
Theorem~$\ref{thm:main}$ holds for every number field $K$ with no
index-two subfield.
\end{proposition}

\begin{proof}
Let $I$ be a fractional ideal of $\OK$, and put $k=\lceil n/2\rceil$.
Since $k\leq n+1-k$, transference gives
\[
 \lambda_k(L(I,1))\lambda_k(L(I,1)^*)\leq T_n.
\]
Thus one of these lattices contains $k$ independent vectors of
length at most $\sqrt{T_n}$. By \eqref{eq:duality}, they correspond
to elements $\alpha_1,\ldots,\alpha_k$ of either $I$ or
$I^\vee=\mathfrak d_K^{-1}I^{-1}$, with bounded normalized lengths.

The field
\[
 E=\Q(\alpha_2/\alpha_1,\ldots,\alpha_k/\alpha_1)
\]
contains the $k$ independent ratios $\alpha_i/\alpha_1$.
Hence $[E:\Q]\geq k$. Since $K$ has no index-two subfield,
every proper subfield has degree at most $n/3<k$.
It follows that $E=K$, and Lemma~\ref{lem:ratios} completes the proof.
\end{proof}

Taking $I=\OK$ allows us to use only the ring of integers and the inverse
different. However, the same proof applies to any fractional ideal $I$
and its dualizing module $I^\vee$.

To treat fields with an index-two subfield, we attempt to use the same
construction for every $I$ and every scaling $u$. Its failure for every $I$ and $u$  
forces a gap in the successive minima. As we now show, this gap
determines a subfield line that is invariant under every unit.
Recall that a CM field is a totally complex quadratic extension
of a totally real field.

\begin{proposition}\label{prop:reduction}
There exists a constant $C_0(n)$ with the following property.
If $K$ has no defining polynomial of height at most $C_0(n)\sqrt D$,
then $n=2m$ and $K$ is a CM field with maximal totally real subfield
$F=K^+$. Moreover, $\OK^\times=\OF^\times$.
For each nonzero fractional ideal $I$, there is a one-dimensional
$F$-subspace $W_I$ of $K$ such that, for every scaling $u$ as in
\eqref{eq:lattice}, the lattice $I\cap W_I$ contains $m$ independent
elements of normalized length~$<\sqrt{2m}$.
\end{proposition}

\begin{proof}
Choose
$R>\max\{1,\sqrt{T_n},T_n/\sqrt n\}$, put $\rho=T_n/R$, and set
$C_0(n)=(2nMR)^n$. Suppose now that $K$ has no defining polynomial of
height at most $C_0(n)\sqrt D$. We claim that
\[
 \lambda_{\lfloor n/2\rfloor+1}(L(I,u))>R
 \qquad\text{for every }I\text{ and }u.
\]
Indeed, if there were $\lfloor n/2\rfloor+1$ independent elements of
normalized length at most $R$, their ratios to the first would remain
independent over $\Q$. The field generated by these ratios would
therefore have degree greater than $n/2$. Since every proper subfield
of $K$ has degree at most $n/2$, these ratios would generate $K$,
contrary to Lemma~\ref{lem:ratios}.

If $n$ were odd, put $j=(n+1)/2$, so that $j=n+1-j$.
The claim applied to both $L(I,u)$ and its dual would give
$\lambda_j(L(I,u))\lambda_{n+1-j}(L(I,u)^*)>R^2>T_n$,
contrary to \eqref{eq:transference}.
We must therefore have $n=2m$. Applying the claim to the dual and
using \eqref{eq:transference} again yields
\begin{equation}\label{eq:gap}
 \lambda_m(L(I,u))<\rho<R<\lambda_{m+1}(L(I,u)),
 \qquad \rho<\sqrt n.
\end{equation}

Fix $I$ and consider the rational space
\[
 W(I,u)=\Span_\Q\{\alpha\in I:\|t_Iu\sigma(\alpha)\|<R\}.
\]
By \eqref{eq:gap}, this space has dimension exactly $m$ for every $u$.
A basis chosen from its short elements remains short when $u$ is
varied sufficiently little. The dimension remains $m$, so the space
itself is locally constant. The set of allowed $u$ is connected,
as is clear after taking logarithms. Hence $W(I,u)$ is independent
of $u$; write it as $W_I$.

Choose a short basis $\alpha_1,\ldots,\alpha_m$ of $W_I$, and let
$E=\Q(\alpha_2/\alpha_1,\ldots,\alpha_m/\alpha_1)$.
Lemma~\ref{lem:ratios} implies that $E\ne K$. On the other hand,
the $m$ independent ratios imply $[E:\Q]\geq m$. It follows that
$[E:\Q]=m$ and $W_I=\alpha_1E$.

For a unit $\varepsilon\in\OK^\times$, the tuple
$(u_v|\sigma_v\varepsilon|)_v$ is another allowed scaling.
Multiplication by $\varepsilon$ carries the short elements for this
new scaling bijectively to the short elements for $u$. Since $W_I$
does not depend on the scaling, we have $\varepsilon W_I=W_I$.
The equality $W_I=\alpha_1E$ now implies $\varepsilon\in E$.
Thus $\OK^\times=\mathcal O_E^\times$, and Dirichlet's unit theorem gives
\[
 m\leq r_1(K)+r_2(K)
   =r_1(E)+r_2(E)\leq [E:\Q]=m.
\]
Equality throughout forces $K$ to be totally complex and $E$ to be
totally real.

Finally, in any fixed embedding $K\subset\C$, a real element of $K$
outside $E$ would generate the quadratic extension $K/E$. This would
make $K$ a subfield of $\R$, a contradiction. Hence
$E=K\cap\R=K^+$, independently of $I$. Writing this field as $F$,
we obtain the asserted $F$-line $W_I$ and the equality
$\OK^\times=\OF^\times$. The bound on the lengths of $m$ independent
elements of $I\cap W_I$ follows from \eqref{eq:gap}.
\end{proof}

\section{The discriminant of the base field}\label{sec:base}

The preceding proposition restricts the possible exceptions to CM
fields. We now obtain a further restriction: the discriminant of the
totally real subfield must be smaller than every fixed positive power
of $D$. The proof compares an upper bound for a relative class group,
obtained by counting ideals, with the analytic class-number formula.

\begin{proposition}\label{prop:small-base}
Under the failure hypothesis of Proposition~\ref{prop:reduction},
the maximal totally real subfield $F=K^+$ satisfies
\begin{equation}\label{eq:small-base}
 D_F\ll_{m,\epsilon}D^\epsilon
 \qquad\text{for every }\epsilon>0.
\end{equation}
\end{proposition}

\begin{proof}
Let $\mathfrak D_{K/F}$ be the relative discriminant ideal, and write
$d=\N(\mathfrak D_{K/F})$, so that $D=D_F^2d$.
Fix a nonzero fractional ideal $I$ of $K$, and let $W_I$ be the line
provided by Proposition~\ref{prop:reduction}.
Choose $0\ne \alpha\in I\cap W_I$, and put $J=(\alpha)I^{-1}$.
Since $\alpha\in I$, the ideal $J$ is integral.

We remove from $J$ the largest ideal extended from $F$ that divides it.
More precisely, define an integral ideal $I_0$ of $F$ by
\[
 v_P(I_0)=\min_{P_1\mid P}
       \left\lfloor\frac{v_{P_1}(J)}{e(P_1/P)}\right\rfloor
\]
for each prime ideal $P$ of $F$, where $e(P_1/P)$ is the ramification
index. We may then write $J=I_0\OK J_0$ with $J_0$ integral.
The fractional ideal $I_1=\alpha^{-1}I\cap F$ equals $I_0^{-1}$.
Indeed, $\alpha^{-1}I=J^{-1}$, and taking valuations gives
\[
 v_P(I_1)=\max_{P_1\mid P}
       \left\lceil-\frac{v_{P_1}(J)}{e(P_1/P)}\right\rceil
       =-v_P(I_0).
\]
Since $W_I=\alpha F$, it follows that $I\cap W_I=\alpha I_1$.

We compute the covolume $V$ of the rank-$m$ lattice
$t_Iu\sigma(\alpha I_1)$ in its real span. At each complex place of $K$,
the real line arising from $\alpha F$ has its length multiplied by
$t_Iu_v|\sigma_v(\alpha)|$. Consequently,
\[
 V=t_I^m
       \left(\prod_{v=1}^m u_v|\sigma_v(\alpha)|\right)
       \sqrt{D_F}\,\N(I_1).
\]
We have $\prod_vu_v=1$ because $K$ is totally complex.
Using also $\N(J)=\N(I_0)^2\N(J_0)$,
$I_1=I_0^{-1}$, $t_I^{2m}=2^m/(\sqrt D\,\N(I))$, and
$\sqrt D=D_F\sqrt d$, we obtain
\begin{equation}\label{eq:volume}
 V^2=2^m\frac{D_F}{\sqrt D}
         \frac{|\N(\alpha)|}{\N(I)}\N(I_1)^2
     =2^m\frac{\N(J_0)}{\sqrt d}.
\end{equation}
By Proposition~\ref{prop:reduction}, this lattice contains $m$
independent vectors of length less than $\sqrt{2m}$.
Their span over $\Z$ is a sublattice of finite index, so its covolume
is an integer multiple of $V$. Hadamard's inequality therefore gives
$V\leq(2m)^{m/2}$. It follows from \eqref{eq:volume} that
$\N(J_0)\leq m^m\sqrt d$.

Modulo classes extended from $F$, the equality
$J=(\alpha)I^{-1}=I_0\OK J_0$ gives $[J_0]=[I]^{-1}$.
As $I$ varies, we therefore obtain an integral representative of
norm at most $m^m\sqrt d$ for every class in
$\Cl(K)/\operatorname{im}\Cl(F)$.
For any number field of degree $n$, the number of integral ideals
of norm at most $x\geq1$ is at most $x(1+\log x)^{n-1}$.
To see this, bound the number of ideals of each norm by the
$n$-fold divisor function and sum over products of $n$ positive
integers at most $x$.
Writing $h_E$ for the class number of a field $E$, we conclude that
\begin{equation}\label{eq:class-count}
 \frac{h_K}{h_F}
 \leq\bigl|\Cl(K)/\operatorname{im}\Cl(F)\bigr|
 \ll_m\sqrt d\,(\log(2d))^{2m-1}.
\end{equation}
The first inequality uses only that the image of $\Cl(F)$ has order
at most $h_F$; injectivity of the extension map is not needed.

We now compare this estimate with the class-number formula.
Let $\kappa_E=\Res_{s=1}\zeta_E(s)$, and let $w_E$ and $R_E$ denote
the number of roots of unity and the regulator of $E$, respectively.
We use the regulator obtained by deleting one coordinate from the
unit logarithms $(d_v\log|\sigma_v\varepsilon|)_v$, with regulator one
when the unit rank is zero.
The equality $\OK^\times=\OF^\times$ implies
$w_K=w_F=2$ and $R_K=2^{m-1}R_F$.
If $\eta=\eta_{K/F}$ is the quadratic Hecke character associated with
$K/F$, then $\zeta_K(s)=\zeta_F(s)L_F(s,\eta)$.
The two class-number formulas consequently give
\[
 \frac{h_K}{h_F}
 =\frac{2}{(2\pi)^m}\sqrt{D_Fd}\,L_F(1,\eta),
 \qquad L_F(1,\eta)=\frac{\kappa_K}{\kappa_F}.
\]

We use the following Brauer--Siegel-type estimate:
for fields $E$ of degree at most $N$ and every $\epsilon>0$,
\begin{equation}\label{eq:residue}
 D_E^{-\epsilon}\ll_{N,\epsilon}\kappa_E
       \ll_{N,\epsilon}D_E^\epsilon.
\end{equation}
The lower bound does not require $E$ to be normal over $\Q$.
Indeed, Stark's residue bound \cite[Lemma~4, p.~140]{Stark} gives
$\kappa_E\gg_N\min\{1-\beta,(\log D_E)^{-1}\}$, where the first
term is omitted if there is no exceptional real zero $\beta$.
If $1-\beta\geq(4N!\log D_E)^{-1}$, this already suffices.
Otherwise Stark's Lemma~8 \cite[p.~145]{Stark}, which applies to non-normal fields as
well, places the same zero in a quadratic subfield $E_0\subset E$.
Siegel's theorem gives $1-\beta\gg_\epsilon D_{E_0}^{-\epsilon}$
after an arbitrarily small adjustment of $\epsilon$.
Since $D_{E_0}\leq D_E$, this proves the required lower bound.
The upper bound follows from the standard estimate
$\kappa_E\ll_N(\log D_E)^{N-1}$. The field $E=\Q$ may be handled
separately, since its residue is one.

Applying \eqref{eq:residue} to $K$ and $F$, with sufficiently small
exponents, gives
$L_F(1,\eta)\gg_{m,\epsilon}D^{-\epsilon/4}$.
Comparison with \eqref{eq:class-count} then yields
\[
 \sqrt{D_F}
 \ll_{m,\epsilon}D^{\epsilon/4}(\log(2D))^{2m-1}
 \ll_{m,\epsilon}D^{\epsilon/2}.
\]
Squaring proves \eqref{eq:small-base}.
\end{proof}

\section{Quadratic norm forms and short vectors}\label{sec:forms}

Let $K/F$ be a CM extension, with $[F:\Q]=m$ and $n=2m$.
Write $\bar z$ for the nontrivial automorphism of $K/F$.
The polynomial constructed in Section~\ref{sec:polynomial} is the
norm of a binary quadratic form over $F$. We first record this
interpretation and its coefficient bound.

\begin{lemma}\label{lem:quadratic-form}
Let $I$ be a nonzero fractional ideal of $K$, and let $\alpha,\beta\in I$ satisfy
$\alpha\ne0$ and $\Q(\beta/\alpha)=K$. If both elements have length at most $B$
in $L(I,u)$, then
\begin{equation}\label{eq:relative-form}
 Q(x,y)=\N_{K/F}(\alpha x-\beta y)
 =\alpha\bar \alpha x^2-(\alpha\bar \beta+\bar \alpha \beta)xy+\beta\bar \beta y^2
\end{equation}
belongs to $F[x,y]$ and is positive definite at every real place.
The polynomial
\begin{equation}\label{eq:form-polynomial}
 f(x)=\frac{\N_{F/\Q}(Q(x,1))}{\N(I)}
\end{equation}
is integral, has degree $n$ and root $\beta/\alpha$, and satisfies
$H(f)\leq(2B^2)^m\sqrt D$.
\end{lemma}

\begin{proof}
The coefficients of $Q$ are fixed by conjugation and hence lie in $F$.
For a real place $v$ of $F$, choose an extension $\sigma_v$ to $K$.
Then $v(Q)$ is
$|\sigma_v(\alpha)x-\sigma_v(\beta)y|^2$, which is positive definite
because $\beta/\alpha$ generates $K$ and hence has no real conjugate.
Transitivity of the norm identifies $f$ with
$\Phi_{I;\alpha,\beta}$ of Section~\ref{sec:polynomial}, proving integrality,
degree, and the assertion about its root.

Put $z_v=t_Iu_v\sigma_v(\alpha)$ and $w_v=t_Iu_v\sigma_v(\beta)$.
Both have modulus at most $B$. The sum of the absolute values of
the coefficients of $Q_v(x,y)=|z_vx-w_vy|^2$ is therefore at most
$4B^2$. Since $\prod_vu_v^2=1$, we have
$f(x)=2^{-m}\sqrt D\prod_v Q_v(x,1)$, and the height bound follows.
\end{proof}

We therefore seek an ideal lattice $L(I,u)$ with $n$ independent
vectors of bounded length. It is enough to find one for which
$\lambda_1(L(I,u))\geq r_m$, where $r_m>0$ depends only on $m$.
Indeed, Minkowski's second theorem gives
\begin{equation}\label{eq:thick-minima}
 r_m^{n-1}\lambda_n(L(I,u))
 \leq\prod_{j=1}^n\lambda_j(L(I,u))
 \leq\frac{2^n}{\omega_n}.
\end{equation}
The resulting independent vectors give a generating ratio by
Lemma~\ref{lem:ratios}, and hence the required quadratic form.

To find such a lattice, we use a function that becomes large
when the lattice has a very short vector. For a Euclidean
lattice $L\subset\R^{2m}$ of covolume one, define
\[
 \theta_L(t)=\sum_{v\in L}e^{-\pi t\|v\|^2},\qquad
 Z_L(s)=\pi^{-ms}\Gamma(ms)
        \sum_{0\ne v\in L}\|v\|^{-2ms}\quad(\Re s>1).
\]
Here $\Gamma$ is the gamma function. The function $Z_L$ is the
completed Epstein zeta function with parameter scaled so that
its central point is $s=1/2$.
Khayutin's lower bound \cite[Corollary~2.8]{Khayutin} implies
the following assertion for a sufficiently small $r_m$.
We give the short estimate from the approximate functional
equation to retain the explicit radius below.

\begin{lemma}\label{lem:epstein-short}
Put
\[
 r_m=\left(1+\frac{(m+4)e^\pi}{4}\right)^{-1/m}.
\]
If $L\subset\R^{2m}$ has covolume one and
$\lambda_1(L)\leq r_m$, then $Z_L(1/2)\geq1$, where the
central value is defined by meromorphic continuation.
\end{lemma}

\begin{proof}
The approximate functional equation of Terras \cite{Terras},
obtained by Poisson summation, gives
\begin{align}
 Z_L(s)={}&-\frac1{ms}-\frac1{m(1-s)}\notag\\
 &+\int_1^\infty
 \bigl((\theta_L(t)-1)t^{ms}
       +(\theta_{L^*}(t)-1)t^{m(1-s)}\bigr)\,\frac{dt}{t};
 \label{eq:epstein-continuation}
\end{align}
see \cite[Theorem~2.3]{Khayutin}, with $L=\Z^{2m}g$,
parameter $2ms$, and $E^*(g,2ms)=Z_L(s)/2$.
The integral is entire in $s$, since the theta series minus one
decay exponentially. The formula gives both
$Z_L(s)=Z_{L^*}(1-s)$ and the reality of $Z_L(1/2)$.
If $v\in L\setminus\{0\}$ has length at most $r<1$, the terms
$v,-v$ contribute at least $2e^{-\pi}$ to $\theta_L(t)-1$
for $1\leq t\leq r^{-2}$. All remaining terms in the integral
at $s=1/2$ are nonnegative, so
\[
 Z_L(1/2)\geq-\frac4m+
       2e^{-\pi}\int_1^{r^{-2}}t^{m/2}\,\frac{dt}{t}
 =-\frac4m+\frac{4e^{-\pi}}m(r^{-m}-1).
\]
For $r=r_m$, the last expression equals $1$.
\end{proof}

In the next section we compute the average of $Z_{L(I,u)}(1/2)$
over ideal classes and scalings. We obtain a bound that tends
to zero for the CM fields remaining after Section~\ref{sec:base}.
For sufficiently large $D$, Lemma~\ref{lem:epstein-short} then
ensures that some lattice has no nonzero vector of length at most $r_m$.

\section{An average over ideal classes}\label{sec:uniform}

The two averages used here had distinct roles. Integrating
over the archimedean scalings converts a sum involving
Euclidean lengths into a sum involving field norms.
Averaging over ideal classes then combines the partial zeta
functions into $\zeta_K$. Thus it suffices to bound this single
Dedekind zeta function.

First, let $K$ be any totally complex field of degree $2m$, and
write $D=D_K$. Choose one embedding $\sigma_v$ from each complex
pair, and put
\[
 T=\{u\in\R_{>0}^m:\textstyle\prod_vu_v=1\},\qquad
 U_K=\{(|\sigma_v(\varepsilon)|)_v:\varepsilon\in\OK^\times\}.
\]
Dirichlet's unit theorem makes $T/U_K$ a compact group; let
$du$ be its Haar probability measure. The lattices
$L(I,u|\sigma(\varepsilon)|)$ and $L(I,u)$ differ by rotations,
so $Z_{L(I,u)}$ is a function on $T/U_K$.
Set
\begin{equation}\label{eq:average-definition}
 \mathcal A_K(s)=\frac1{h_K}\sum_{[I]\in\Cl(K)}
                  \int_{T/U_K}Z_{L(I,u)}(s)\,du .
\end{equation}
Changing $I$ to $aI$ replaces $u$ by the tuple
$u'_v=u_v|\sigma_v(a)|/|\N(a)|^{1/(2m)}$, up to rotations.
Since $u\mapsto u'$ is a translation of $T/U_K$, the average
is independent of the ideal representatives.

\begin{proposition}\label{prop:hecke-average}
For every totally complex field $K$ of degree $2m$, Hecke's
formula gives
\begin{equation}\label{eq:hecke-period}
 \mathcal A_K(s)=\frac{w_K}{m\,h_KR_K}\Lambda_K(s),\qquad
 \Lambda_K(s)=D^{s/2}\bigl((2\pi)^{-s}\Gamma(s)\bigr)^m\zeta_K(s)
\end{equation}
for $s\ne0,1$, with meromorphic continuation understood.
In particular,
\begin{equation}\label{eq:hecke-center}
 \mathcal A_K(1/2)
 =\frac{2^{m/2}\pi^m}{m}\,
       D^{-1/4}\frac{\zeta_K(1/2)}{\kappa_K},
 \qquad \kappa_K=\Res_{s=1}\zeta_K(s).
\end{equation}
\end{proposition}

\begin{proof}
We apply \cite[Theorem~3.4]{Khayutin}, a form of Hecke's
formula \cite{Hecke}. The following observations identify its
normalizations with ours.
\begin{itemize}
\item By Definitions~2.1--2.2 and Lemma~3.5 of \cite{Khayutin},
the Euclidean structure is the same as ours, and
$E^*(g,2ms)=Z_L(s)/2$ when $L=\Z^{2m}g$ has covolume one.
\item The compact group
$[H]=\OK^\times\backslash((\C^\times)^m/\R^\times)$ in
Definitions~3.1--3.2 maps onto $T/U_K$ by
$h\mapsto(|h_v|/(\prod_w|h_w|)^{1/m})_v$.
Its kernel consists of rotations modulo units.
This map carries Haar probability measure to $du$, and the
corresponding normalized ideal lattices are isometric.
\item The completed partial zeta function in Definition~3.3 is
\[
 \zeta_I^*(s)=D^{s/2}
       \bigl((2\pi)^{-s}\Gamma(s)\bigr)^m
       \zeta_K(s,[I]^{-1}),
\]
where $\zeta_K(s,c)$ sums $\N(J)^{-s}$ over the integral ideals
$J$ in the class $c$. Indeed, $\alpha\mapsto(\alpha)I^{-1}$
identifies nonzero elements of $I$ modulo units with these
ideals.\footnote{The sum in \cite[Definition~3.3]{Khayutin}
is understood modulo units, as in the proof of \cite[Theorem~3.4]{Khayutin}.}
\item The regulator uses $2\log|\sigma_v(\varepsilon)|$ at
complex places, as in the proof of
\cite[Proposition~4.1]{Khayutin}, and thus agrees with $R_K$.
\end{itemize}
The cited theorem therefore gives
\[
 \int_{T/U_K}Z_{L(I,u)}(s)\,du
 =2\frac{w_K}{2mR_K}\zeta_I^*(s).
\]
Averaging over the ideal classes proves \eqref{eq:hecke-period},
including its continuation to $s\ne0,1$.
The class-number formula
$\kappa_K=(2\pi)^m h_KR_K/(w_K\sqrt D)$ and
$\Gamma(1/2)=\sqrt\pi$ give \eqref{eq:hecke-center}.
\end{proof}

We now return to a CM quadratic extension $K/F$ with $[F:\Q]=m$.
Let $\eta=\eta_{K/F}$ be its quadratic Hecke character and put
$d=\N(\mathfrak D_{K/F})$, so that $D=D_F^2d$.
The character has finite conductor $\mathfrak D_{K/F}$ and
equals the sign character at every real place of $F$.
Its $L$-function satisfies
$\zeta_K(s)=\zeta_F(s)L_F(s,\eta)$.

\begin{lemma}\label{lem:uniform-subconvexity}
There are $a\geq1/2$ and $0<\delta<1/4$, depending only on $m$,
such that
\begin{equation}\label{eq:zeta-subconvexity}
 |\zeta_K(1/2)|\ll_m D_F^a d^{1/4-\delta}.
\end{equation}
\end{lemma}

\begin{proof}
Apply Michel--Venkatesh \cite[Theorem~1.2]{MV} to the
Eisenstein representations
$\pi_1=1\boxplus\eta$ and
$\pi_2=|\cdot|^{iv}\boxplus|\cdot|^{-iv}$, with $v\ne0$
tending to zero. Here $\boxplus$ denotes the representation
induced from the displayed pair of Hecke characters; its
$L$-function is the product of their $L$-functions.
Both representations have unitary central character.
Since $\eta\ne1$, $\pi_1$ is nonsingular, and
\[
 L(\pi_1\otimes\pi_2,1/2)
   =\zeta_K(1/2+iv)\zeta_K(1/2-iv).
\]
The finite conductor of this product has norm $d^2$;
the archimedean parameters are bounded for $|v|\leq1$.
The second representation is unramified at every finite place.
The polynomial dependence stated after (1.1) of \cite{MV}
therefore gives, uniformly for such $v$,
\[
 |\zeta_K(1/2+iv)\zeta_K(1/2-iv)|
       \ll_m D_F^b d^{1/2-2\delta}
\]
for some $b\geq0$.\footnote{Immediately after (1.1), they write:
\textquotedblleft more precisely, the constant implied depends polynomially on the
discriminant of $F$ (for $F$ varying over fields of given degree)
and on $C(\pi_2)$.\textquotedblright\ Their
\S\S2.1.1 and~4.1.5 explain the convention for polynomial
implicit constants. Any base-discriminant factor included in
the analytic conductor is absorbed into $D_F^b$.}
Letting $v\to0$ and taking a square root proves
\eqref{eq:zeta-subconvexity}; we may enlarge $a$ to ensure $a\geq1/2$.
\end{proof}

\begin{proposition}\label{prop:small-average}
There are $A\geq0$ and $\delta>0$, depending only on $m$, such
that for every $\epsilon>0$,
\begin{equation}\label{eq:small-average}
 |\mathcal A_K(1/2)|
       \ll_{m,\epsilon}D_F^A D^{-\delta+\epsilon}.
\end{equation}
The implied constant may be ineffective.
\end{proposition}

\begin{proof}
Combine \eqref{eq:hecke-center}, \eqref{eq:zeta-subconvexity},
and the residue bound \eqref{eq:residue} to obtain
\[
 |\mathcal A_K(1/2)|
 \ll_{m,\epsilon} D^{-1/4+\epsilon}D_F^a d^{1/4-\delta}
 =D_F^{a-1/2+2\delta}D^{-\delta+\epsilon}.
\]
This gives the assertion with $A=a-1/2+2\delta\geq0$.
\end{proof}

\section{Proof of Theorem~\ref{thm:main}}\label{sec:conclusion}

As noted in Section~\ref{sec:polynomial}, the case $n=1$ is immediate.
For $n\geq2$, fields with a defining polynomial of height at most
$C_0(n)\sqrt D$ already satisfy the theorem.
For every remaining field, Propositions~\ref{prop:reduction}
and~\ref{prop:small-base} imply that $n=2m$, that $K$ is CM with
maximal totally real subfield $F$, and that
$D_F\ll_{m,\gamma}D^\gamma$ for every $\gamma>0$.

Take $\epsilon=\delta/4$ and
$\gamma=\delta/(4(A+1))$ in Proposition~\ref{prop:small-average}
and Proposition~\ref{prop:small-base}, respectively. Then
\[
 |\mathcal A_K(1/2)|\ll_m D^{-\delta/2}.
\]
For $D$ sufficiently large in terms of $m$, this is less than $1$.
If every lattice in the average \eqref{eq:average-definition}
had a nonzero vector of length at most $r_m$,
Lemma~\ref{lem:epstein-short} would give
$\mathcal A_K(1/2)\geq1$. Therefore some ideal $I$ and some
scaling $u$ satisfy $\lambda_1(L(I,u))>r_m$.

By \eqref{eq:thick-minima}, this lattice has $n$ independent
vectors of bounded length. The corresponding elements
$\alpha_1,\ldots,\alpha_n\in I$ have ratios to $\alpha_1$
spanning $K$ over $\Q$. Lemma~\ref{lem:ratios} therefore supplies
$\alpha,\beta\in I$, both of bounded normalized length, such
that $\beta/\alpha$ generates $K$.
Lemma~\ref{lem:quadratic-form} now gives the desired polynomial
\[
 f(x)=\frac{\N_{F/\Q}(Q(x,1))}{\N(I)},
 \qquad Q(x,y)=\N_{K/F}(\alpha x-\beta y),
 \qquad H(f)\ll_n\sqrt D.
\]
All constants here depend only on $n$.
Finally, Hermite's theorem leaves only finitely many degree-$n$
fields below the discriminant threshold. Choosing a defining
polynomial for each and enlarging the constant absorbs these
fields and completes the proof.

\subsection*{Acknowledgments}

We are very grateful to Akshay Venkatesh for many helpful conversations and, in particular, for explaining to us the results in~\cite{MV} and \cite{Venkatesh}.  The first-named author 
was supported by the National Science Foundation (Grant DMS-1001828).

\end{document}